\documentclass[a4paper]{amsart}

\usepackage{amsmath,amsthm,amssymb,amscd}
\usepackage[arrow,matrix]{xy}
\usepackage[dvips]{graphicx}
\usepackage{enumerate}

\theoremstyle{plain}
\numberwithin{equation}{section}
\newtheorem{thm}{Theorem}[section]

\newtheorem{prop}[thm]{Proposition}

\newtheorem{lem}[thm]{Lemma}
\theoremstyle{definition}
\newtheorem{dfn}[thm]{Definition}
\newtheorem{exm}[thm]{Example}

\newtheorem{rem}[thm]{Remark}

\def\rank{\mathop{\mathrm{rank}}\nolimits}
\def\dim{\mathop{\mathrm{dim}}\nolimits}
\def\Lie{\mathop{\mathrm{Lie}}\nolimits}

\def\kutorsor{E_{\sigma}\to\Gamma(\sigma)^{\mathrm{gp}} \backslash D_{\sigma}}
\def\torsor{E_{\sigma}\to\Gamma(\sigma)^{\mathrm{gp}} \backslash D_{\sigma}}
\def\gsds{\Gamma(\sigma)^{\mathrm{gp}} \backslash D_{\sigma}}
\def\gs{\Gamma(\sigma)^{\mathrm{gp}}}
\def\GsDs{\Gamma(\sigma)^{\mathrm{gp}} \backslash D_{\sigma}}
\def\Gs{\Gamma(\sigma)^{\mathrm{gp}}}
\def\gm{\mathbb{G}_{m}}

\def\es{E_{\sigma}}
\def\Es{E_{\sigma}}

\def\CC{\mathbb{C}}
\def\QQ{\mathbb{Q}}
\def\RR{\mathbb{R}}
\def\ZZ{\mathbb{Z}}
\def\PP{\mathbb{P}}
\def\Gr{\mathrm{Gr}}
\def\fix{\mathrm{fix}}

\def\calB{\mathcal{B}}
\def\calM{\mathcal{M}}

\def\frakg{\mathfrak{g}}

\def\nn{\mathbf{n}}
\def\hh{\mathbf{h}}

\def\bs{\backslash}
\def\gp{\mathrm{gp}}
\def\hom{\mathop{\mathrm{Hom}}\nolimits}
\def\Aut{\mathop{\mathrm{Aut}}\nolimits}
\def\torus{\mathrm{torus}}
\def\toric{\mathrm{toric}}
\def\spec{\mathop{\mathrm{Spec}}\nolimits}
\def\Im{\mathop{\mathrm{Im}}\nolimits}
\def\even{\text{even}}
\def\odd{\text{odd}}
\def\l{\left}
\def\r{\right}

\newcommand{\mf}[1]{{\mathfrak{#1}}}
\newcommand{\mb}[1]{{\mathbf{#1}}}

\title[moduli of log Hodge structures, II]{On the boundary of moduli spaces\\ of log Hodge structures, II:\\ nontrivial torsors}
\author[T.~Hayama]{Tatsuki HAYAMA}
\address{Department of Mathematics, National Taiwan University, Taipei 106, Taiwan}
\curraddr{Mathematical Sciences Center, Tsinghua University, Haidian District, Beijing 100084, China}
\email{tatsuki@math.tsinghua.edu.cn}
\thanks{Supported by National Science Council of Taiwan}
\date{\today}

\begin{document}
\maketitle
\begin{abstract}
In this paper we determine when a natural torsor arising in the work \cite{ku} of Kato and Usui on partial compactification of period  domains of pure Hodge structure is trivial, and give an application to cycle spaces.
\end{abstract}
\section{Introduction}
Let $D$ be a period domain of pure Hodge structures defined by Griffiths \cite{G}.
A variation of $\ZZ$-Hodge structure over the $n$-product of punctured disk $(\Delta^*)^n$ gives the period map $(\Delta^*)^n\to \Gamma\bs D$ where $\Gamma$ is the monodromy group, i.e.\ the $\ZZ$-module generated by the monodromy transformations.
We assume the monodromy transformations are unipotent.
In this paper we treat partial compactifications of $\Gamma\bs D$ so that the period map is extended over $\Delta^n$.

In the case where $D$ is Hermitian symmetric, Ash, Mumford, Rapoport and Tai \cite{amrt} give partial compactifications of $\Gamma\bs D$ (and also give compactifications of arithmetic quotient of $D$).
Later, Kato and Usui \cite{ku} generalize toroidal partial compactifications for any period domain $D$, which is not Hermitian symmetric in general, and show these are moduli spaces of log Hodge structures. 
This partial compactification is given by using toroidal embedding associated to the cone generated by the data of the monodromy.
In fact, for generators $T_1,\ldots ,T_n$ of the monodromy group $\Gamma$, the partial compactification $\Gamma\bs D_{\sigma}$ is given by the cone $\sigma=\sum_{j=1}^n\RR_{\geq 0}N_j$ ($N_j=\log{T_j}$) in the Lie algebra.
Here a boundary point is a nilpotent orbit associated to a face of $\sigma$ (see \S \ref{def-log}). 

In the ``classical situation" (i.e. $D$ is Hermitian symmetric), $\Gamma\bs D_{\sigma}$ is an analytic space.
In contrast, for general period domains, $\Gamma\bs D_{\sigma}$ may not be an analytic space.
In fact the boundary components of the partial compactification $\Gamma\bs D_{\sigma}$ can have codimension greater than $1$ although it is $1$ in the classical situation (see Example \ref{1111}). 
This mean there can be {\it slits} on the boundary of $\Gamma\bs D_{\sigma}$. 
Kato and Usui \cite{ku} define logarithmic manifolds as a generalization of analytic spaces and state $\Gamma\bs D_{\sigma}$ is a logarithmic manifold.
 
\subsection*{}
A part of the geometric structure of $\Gamma\backslash D_{\sigma}$ is given by a torsor $E_{\sigma}\to \Gamma\backslash D_{\sigma}$ constructed in [KU] (see 3.1).
We discuss about these torsors.
Our main result is following:
\begin{thm}[Theorem \ref{no_global_sec}]
$\Es\to\Gamma\bs D_{\sigma}$ is trivial if and only if $D$ is Hermitian symmetric or $\sigma=\{0\}$.
\end{thm}
In \cite{H}, we proved this result in the case where $D$ is Hermitian symmetric.
In this paper, we treat the case where $D$ is not hermitian symmetric. 
To prove Theorem 1.1, we use a result from the book of Fels, Huckleberry and Wolf \cite{FHW} to show that, unless D is Hermitian symmetric, any holomorphic function on $D$ is constant.

We used a different strategy to prove the non-triviality of the torsors for the one-example in \cite[Proposition 5.8]{H}.
Generalizing this approach, we give another proof of the non-triviality result in Proposition \ref{no_loc_sec}.
This second proof is stronger than the first one since it gives a non-triviality on some open sets around a boundary point.
We use cycle spaces and the SL(2)-orbit theorem there.
Some property of cycle spaces induces the non-triviality.

In conclusion, the properties of cycle spaces induce the above non-triviality results.
Cycle spaces can have a significance on the study of moduli spaces of log Hodge structures in wider framework. 
In fact, the property of cycle spaces of Lemma \ref{fix} induces our later work \cite{H2}. 
On the other hand, Green, Griffiths and Kerr \cite{GGK2} have introduced Mumford-Tate domains as a generalization of period domains, and they also indicate the importance of cycle space concerning about the cohomology groups of Mumford-Tate domains.
Moreover Kerr and Pearlstein \cite{KP} have constructed partial compactifications of Mumford-Tate domains in the same manner as Kato and Usui \cite{ku}.
We expect that our results fit into the case for the boundaries of the Mumford-Tate domains.
 
\subsection*{}
This paper is organized as follows: In \S 2.1, we review period domains of Hodge structures.
In \S 2.2 and \S 2.3, we discuss cycles spaces of period  domains.  In \S 3, we review
moduli spaces of polarized log Hodge structures.  In \S 4, we reformulate our priori results 
of [H] in terms of cycle spaces.

\section{Cycle spaces of period domains}
\subsection{Polarized Hodge structures and period domains}\label{pd}
We recall the definition of polarized Hodge structures and of period domains.
A Hodge structure of weight $w$ with Hodge numbers $(h^{p,q})_{p,q}$ is a pair $(H_{\ZZ},F)$ consisting of a free $\ZZ$-module of rank $\sum_{p,q}h^{p,q}$ and of a decreasing filtration on $H_{\CC}:=H_{\ZZ}\otimes \CC$ satisfying the following conditions:
\begin{enumerate}
\item[(H1)] $\dim_{\CC} F^p=\sum_{r\geq p}h^{r,w-r}\quad \text{for all $p$;}$\label{check1}
\item[(H2)] $H_{\CC}=\bigoplus _{p+q=w} H^{p,q}\quad(H^{p,q}:=F^p\cap \overline{F^{w-p}}).$
\end{enumerate}

A polarization $\langle\; ,\;\rangle$ for a Hodge structure $(H_{\ZZ},F)$ of weight $w$ is a non-degenerate bilinear form on $H_{\QQ}:=H\otimes \QQ$, symmetric if $w$ is even and skew-symmetric if $w$ is odd, satisfying the following conditions:
\begin{enumerate}
\item[(P1)] $\langle F^p , F^q \rangle=0\quad \text{for $p+q>w$;}$\label{check2}
\item[(P2)] $i^{p-q}\langle v,\bar{v}\rangle >0$ for $0\neq v\in H^{p,q}$.
\end{enumerate}
 
We fix a polarized Hodge structure $(H_{\ZZ,0},F_0,\langle\; ,\;\rangle_0)$ of weight $w$ with Hodge numbers $(h^{p,q})_{p,q}$.
We define the set of all Hodge structures of this type
$$D:=\left\{\begin{array}{l|l}F&\begin{array}{r}(H_{\ZZ,0} , F ,\langle\; ,\;\rangle_0)\text{ is a polarized Hodge structure}\\ \text{ of weight $w$ with Hodge numbers $(h^{p,q})_{p,q}$}\end{array}\end{array}\right\}.$$
$D$ is called a period domain.
Moreover, we have the flag manifold 
$$\check{D}:=\left\{\begin{array}{l|l}F&\begin{array}{r}(H_{\ZZ,0} , F ,\langle\; ,\;\rangle_0)\text{ satisfies the conditions}\\ \text{ (H1), (H2) and (P1)}\end{array}\end{array}\right\}.$$
$\check D$ is called the compact dual of $D$, and contains $D$ as an open subset.
Let $G_A:=\Aut{(H_{A,0},\langle\; ,\;\rangle_0)}$.
Then, $G_{\RR}$ acts transitively on $D$ and $G_{\CC}$ acts transitively on $\check D$.
$G_{\RR}$ is a classical group such that
\begin{align*}
G_{\RR}\cong\begin{cases}
Sp(h,\RR)& \text{if $w$ is odd,}\\
SO(h_{\text{odd}},h_{\text{even}})& \text{if $w$ is even,}\\
\end{cases}
\end{align*}
where $2h=\rank{H_{\ZZ}}$, $Sp(h,\RR)$ is the ($2h\times 2h$)-matrix symplectic group, $h_{\text{odd}}=\sum_{\text{p:odd}}h^{p,q}$ and $h_{\text{even}}=\sum_{\text{p:even}}h^{p,q}$.

Let $\mathfrak{g}_{A}=\Lie{G_{A}}$ ($A=\RR,\CC$).
We then have the decomposition $\mathfrak{g}_{\CC}=\bigoplus_{p+q=0} \mathfrak{g}^{p,q}$ given by
$$\mathfrak{g}^{p,q}=\left\{ \alpha\in \mathfrak{g}_{\CC} \;|\; \alpha H^{p',q'}\subset H^{p+p',q+q'}\text{ for }p',q'\in\ZZ \right\}$$
with respect to a Hodge decomposition $H_{\CC}=\bigoplus H^{p,q}$.
\begin{exm}[Upper half plane]\label{SL(2)}
Let us consider the case where the Hodge numbers $h^{1,0}=h^{0,1}=1$, $0$ otherwise.
Then corresponding classifying space $D$ is the upper-half plane $\{z\in\CC\; |\; \Im{z}>0\}$, and $\check{D}\cong\PP^1$.
$G_{A}\cong SL(2,A)$ $(A=\ZZ,\RR,\CC)$ where the action of $SL(2,\CC)$ on $\check{D}$ is given by the linear fractional transformation.
Here $\frakg_{\RR}=\frak{sl}(2,\RR)$ is generated by
\begin{align*}
\nn _- =\begin{pmatrix}
0&1\\0&0
\end{pmatrix},\quad
\hh =\begin{pmatrix}
-1&0\\ 0&1
\end{pmatrix},\quad
\nn_+ =
\begin{pmatrix}
0&0\\ 1&0
\end{pmatrix}.
\end{align*}
We call the triple the $sl_2$-triple.
The $sl_2$-triple satisfies the following conditions:
\begin{align*}
[\nn _{+} ,\nn_{-}]= \hh ,\quad [\nn_{\pm},\hh ]=\pm 2\nn_{\pm}. 
\end{align*}
The Hodge decomposition of $\frakg_{\CC}$ with respect to $i\in D$ is given by
\begin{align}\label{sl2-triple}
\frakg^{-1,1}=\CC (i\nn_- -\hh +i\nn_+ ), \quad \frakg^{0,0}=\CC (\nn_- -\nn_+),\quad \frakg^{1,-1}=\overline{\frakg^{-1,1}}.
\end{align}
\end{exm}

Returning to the general case, the isotropy subgroup $L$ of $G_{\RR}$ at $F_0$ is given by
\begin{align*}
L&=\{ g\in G_{\RR}\; |\; gF_0=F_0\}\\
&\cong
\begin{cases}
\prod_{p\leq m}U(h^{p,q})& \text{if $w=2m+1$,}\\
\prod_{p<m}U(h^{p,q})\times SO(h^{m,m})& \text{if $w=2m$.}\\
\end{cases}
\end{align*}
They are compact subgroups of $G_{\RR}$ but not maximal compact unless $D$ is Hermitian symmetric.
We define
$$H^{\text{even}}=\bigoplus_{p : \text{even}}H^{p,q}_0,\quad H^{\text{odd}}=\bigoplus_{p : \text{odd}}H^{p,q}_0$$
where $H_{\CC}=\bigoplus H^{p,q}_0$ is the Hodge decomposition for $F_0$.
Here 
\begin{align*}
K&=\{ g\in G_{\RR}\;|\; gH^{\text{even}}=H^{\text{even}} \}\\
&\cong
\begin{cases}
U(h)& \text{if $w$ is odd,}\\
S(O(h_{\text{odd}})\times O(h_{\text{even}}))& \text{if $w$ is even}\\
\end{cases}
\end{align*}
is the maximal subgroup containing $L$ (cf. \cite[Example 4.3.6]{CMP}, \cite[Lemma 2.8]{LS}).
By the connectivity of $G_{\RR}$, $D$ is connected if $w$ is odd, $D$ has two connected component if $w$ is even and $h_{\text{even}},h_{\text{odd}}>0$.
Here $D$ is Hermitian symmetric if and only if the isotropy subgroup is a maximally compact subgroup, i.e., one of the following is satisfied (cf. \cite[(1.8)]{U2}):
\begin{enumerate}
\item $w=2m+1, \;h^{p,q}=0 \text{ unless }p=m+1,m;$
\item $w=2m, \;h^{p,q}=1 \text{ for }p=m+1,m-1,\;h^{m,m}\text{ is arbitary, }h^{p,q}=0\text{ otherwise; }$
\item $w=2m, \;h^{p,q}=1 \text{ for }p=m+a,m+a-1,m-a,m-a+1\text{ for some }a\geq 2,\;h^{p,q}=0\text{ otherwise. }$
\end{enumerate} 
In the case (1), $D$ is a Hermitian symmetric domain of type {\rm III}.
In the case (2) or (3), an irreducible component of $D$ is a Hermitian symmetric domain of type {\rm IV}.
We call the cases (1)--(3) the classical situation.
\begin{exm}[The weight $1$ case]\label{weight-1}
We give an example of period domains of weight $1$ and $h^{1,0}=h^{0,1}=n$, $0$ otherwise.
This case is corresponding to the case (1) above.
Now $G_{A}=Sp(n,A)$ ($A=\ZZ,\RR,\CC$) and
\begin{align}\label{siegel}
D&= \left\{\begin{array}{l|l}W\text{ :$\langle \; ,\;\rangle$-isotropic $n$-planes}&\begin{array}{r} W> 0 \text{ for }i\langle \bullet ,\bar{\bullet} \rangle\end{array}\end{array}\right\}\\
&\cong\{Z\in \CC^{n\times n}\; |\; I-ZZ^*> 0\} \nonumber\\
&\cong Sp(n,\RR )/U(n)\nonumber
\end{align}
where ``$>0$" means positive definite.
$D$ is called the Siegel space of degree $n$.
See \cite{n} for detail.
\end{exm}
\subsection{Cycle spaces of period domains}\label{cs-def}
Let $D_0$ be a irreducible component including $F_0$ of a period domain $D$.
Then the identity component $G_{\RR ,0}$ acts on $D_0$ transitively.
Let $K_0$ be the maximal compact subgroup of $G_{\RR, 0}$ containing the isotropy subgroup $L_0$ at $F_0$.
We then have the real analytic projection
$$p:D_0\cong G_{\RR,0}/L_0\to G_{\RR,0}/K_0.$$
If $w$ is odd, 
\begin{align}\label{odd-weight}
G_{\RR,0}/K_0=Sp(h,\RR) /U(h)
\end{align}
is the Hermitian symmetric domain of Example \ref{weight-1}.
If $w$ is even, 
$$G_{\RR,0}/K_0=SO_0(h_{\odd},h_{\even})/SO(h_{\odd})\times SO(h_{\even})$$
is a symmetric space which does not have any complex structure unless the projection $p$ is trivial (cf. \cite[Example 4.3.6]{CMP}).
Moreover, $G_{\RR,0}/K_0$ is written as the set of all $H_{\even}$ for $F\in D_0$ (cf. \cite[Lemma 2.10]{LS}) and the projection $p$ is given by 
$$F\mapsto H_{\even}.$$
Then it is not holomorphic even if $w$ is odd. 
Therefore we have the following theorem applying \cite[Theorem 4.4.3]{FHW} to $D_0$:
\begin{thm}\label{const}
If $D_0$ is not Hermitian symmetric domain (i.e., $L_0 \neq K_0$) any holomorphic function on $D_0$ is constant.
\end{thm}

Now the fiber of $p(F_0)$ is the $K_0$-orbit $C_0=K_0\cdot F_0$.
We call $C_0$ the base cycle of $F_0$.
By \cite[Theorem 4.3.1]{FHW}, $K_{0, \CC}$ acts on $C_0$ transitively, and then $C_0=K_{0, \CC}\cdot F_0$ is a compact submanifold of $D_0$.
\begin{prop}[{\cite[Lemma 5.1.3]{FHW}}]
Let $J=\{g\in G_{\CC}\;|\; gC_0=C_0\}$.
Then $J$ is a closed complex subgroup of $G_{\CC}$.
The quotient manifold $\mathcal{M}_{\check{D}}=\{gC_0\; |\; g\in G_{\CC}\}\cong G_{\CC}/J$ has a natural structure of $G_{\CC}$-homogeneous complex manifold, and the subset $\mathcal\{gC_0\;|\; g\in G_{\CC}\text{ and }gC_0\subset D\}$ is open in $\calM_{\check{D}}$. 
\end{prop}
The topological component of $C_0$ in $\{gC_0\;|\; g\in G_{\CC}\text{ and }gC_0\subset D_0\}$ is called the cycle space of $D_0$.
We denote the cycle space of $D_0$ by $\calM_{D_0}$.
If $D_0$ is Hermitian symmetric, the projection $p$ is trivial, therefore $\calM_{D_0}=D_0$.
\subsection{Cycle spaces for odd-weight cases}\label{cs-exm}
We describe cycle spaces explicitly in the odd-weight case when $D_0$ is not Hermitian symmetric according to \cite[5.5B]{FHW}.
In this case $D=D_0$.
For a base point $F_0\in D$ we define
\begin{align*}
f^p_{\text{even}}=\sum_{\substack{r\geq p,\\ r\text{: even}}}h^{r,s},\quad f^p_{\text{odd}}= \sum_{\substack{r\geq p,\\ r\text{: odd}}}h^{r,s}.
\end{align*}
Let $V$ and $W$ be $\langle \; ,\;\rangle$-isotropic subspaces, and let
$$C_{V,W}=\{ F\in \check{D} |\;\dim{(F^p\cap V)}=f^p_{\even},\dim{(F^p\cap W)}=f^p_{\odd}\}.$$
Here $gC_{V,W}=C_{gV,gW}$ for $g\in G_{\CC}$ by the definition.
Now 
\begin{align*}
&F^p_0\cap H_{\even}=\bigoplus_{\substack{r\geq p,\\ r\text{: even}}}H^{r,s}\\
&gF^p_0\cap H_{\even}=g(F^p_0\cap H_{\even})=\bigoplus_{\substack{r\geq p,\\ r\text{: even}}}gH^{r,s}
\end{align*}
for $g\in K$.
Then $C_0=C_{H^{\even},H^{\odd}}$.

By (\ref{odd-weight}), $G_{\RR}/K$ is isomorphic to the Siegel space $\calB$.
In this case, the cycle space $\calM_{D}$ is described as follows:
\begin{prop}\label{cs}
\begin{align*}
\calM_{D}=\{C_{V,W}|\; V< 0\text{ and }W> 0\text{ for }i^w\langle \bullet ,\bar{\bullet} \rangle \}\cong \calB \times \bar{\calB}
\end{align*}
where $w$ is the weight.
\end{prop}
\begin{proof}
Now the $G_{\RR}$-orbit $G_{\RR}H^{\even}$ is isomorphic to $\calB$.
Then $G_{\RR}H^{\odd}$ is the complex conjugate $\bar{\calB}$.
Since $\calB$ (resp. $\bar{\calB}$) is an open subset of the flag manifold $G_{\CC}H^{\even}$ (resp. $G_{\CC}H^{\odd}$), we have $\calB\times \bar{\calB} \subset G_{\CC}H^{\even}\times G_{\CC}H^{\odd}$.
For $(H^{\even},H^{\odd})\in G_{\CC}H^{\even}\times G_{\CC}H^{\odd}$, the $G_{\CC}$-orbit $G_{\CC}(H^{\even},H^{\odd})$ includes $\calB\times \bar{\calB}$ by \cite[Lemma 5.4.1]{FHW}.
Now the isotropy subgroup of $G_{\CC}$ at $(H^{\even},H^{\odd})$ is $K_{\CC}$.
Then we have  
 $G_{\CC}/K_{\CC}\supset \calB\times\bar{\calB}$.

Since $\calM_{\check{D}}\cong G_{\CC}/J$ and $K_{\CC}\subset J$, we have the projection
$$\pi : G_{\CC}/K_{\CC}\to \calM_{\check{D}};\quad g\pmod{K_{\CC}}\mapsto gC_0.$$
By \cite[Proposition 5.4.3]{FHW}, $\pi$ is injective on $\calB\times \bar{\calB}$.
Moreover, by \cite[Theorem 5.5.1]{FHW}, $\pi(\calB\times \bar{\calB})=\calM_D\subset \calM_{\check{D}}$.
\end{proof}
\section{Moduli spaces of polarized log Hodge structures}
In this section, we review the construction of moduli spaces of log Hodge structures and state the fundamental properties following \cite{ku} in \S \ref{def-log}.
We state the main result in \S \ref{main-sec}

\subsection{Construction and fundamental properties}\label{def-log}
We call $\sigma\subset \frakg_{\RR}$ a nilpotent cone if it satisfies the following conditions:
\begin{enumerate}
\item $\sigma$ is a closed cone generated by finitely many elements of $\frakg_{\QQ}$;
\item $N\in \sigma$ is a nilpotent as an endmorphism of $H_{\RR}$;
\item $NN'=N'N$ for any $N,N'\in \sigma$.
\end{enumerate} 
For $A=\RR,\CC$, we denote by $\sigma_{A}$ the $A$-linear span of $\sigma$ in $\mf{g}_A$.
\begin{dfn}\label{nilp}
Let $\sigma=\sum_{j=1}^n\RR_{\geq 0}N_j$ be a nilpotent cone and $F\in\check{D}$. 
Then 
$$\exp{(\sigma_{\CC})}F\subset\check{D}$$
 is called a $\sigma$-nilpotent orbit if it satisfies the following conditions:
\begin{enumerate}
\item $\exp{(\sum_j iy_jN_j)}F\in D$ for all $y_j\gg 0$.
\item $NF^p\subset F^{p-1}$ for all $p\in\ZZ$ and for all $N\in \sigma$.
\end{enumerate}
\end{dfn}
We define the set of nilpotent orbits
$$D_{\sigma}:=\{(\tau,Z)|\; \tau \text{: face of }\sigma,\; Z\;\text{is a}\;\tau \text{-nilpotent orbit}\}.$$
For a nilpotent cone $\sigma$, we have the abelian group and the monoid
\begin{align*}
\Gs=\exp{(\sigma_{\RR})}\cap G_{\ZZ},\quad \Gamma(\sigma)=\exp{(\sigma)}\cap G_{\ZZ}.
\end{align*}

We define a geometric structure on $\Gs\bs D_{\sigma}$.
First we review some basic facts about toric varieties.
The monoid $\Gamma(\sigma)$ defines the toric varieties
\begin{align*}
&\toric_{\sigma}:=\spec(\CC[\Gamma(\sigma)^{\vee}])_{\mathrm{an}}\cong \hom{(\Gamma(\sigma)^{\vee},\CC)},\\
&\torus_{\sigma}:=\spec(\CC[{\Gamma(\sigma)}^{\vee \mathrm{gp}} ])_\mathrm{an}\cong\hom{({\Gamma(\sigma)}^{\vee \mathrm{gp}},\gm)}\cong \gm\otimes {\Gamma(\sigma)}^{\mathrm{gp}},
\end{align*}
where $\CC$ in the right hand side of the first line is regarded as a semigroup via multiplication and above homomorphisms are of semigroups. 
As in \cite[\S 2.1]{F}, we choose for a face $\tau$ of $\sigma$ the distinguished point 
$$x_{\tau}:\Gamma(\sigma)^{\vee}\to\CC;\quad u\mapsto \begin{cases}1&\text{if }u\in \Gamma(\tau)^{\perp },\\0&\text{otherwise.}\end{cases}$$
Then $\toric_{\sigma}$ can be decomposed by torus orbits as 
$$\toric_{\sigma}=\bigsqcup_{\tau\text{: face of }\sigma}(\torus_{\sigma}\cdot x_{\tau}).$$
For $q\in\toric_{\sigma}$, there exists the face $\sigma(q)$ of $\sigma$ such that $q\in\torus_{\sigma}\cdot x_{\sigma(q)}$.
By a surjective homomorphism
\begin{equation*}
\mb{e}:\sigma_{\CC}\rightarrow \torus_{\sigma}\cong\gm\otimes {\Gamma(\sigma)}^{\mathrm{gp}};\;w\log{(\gamma)}\mapsto \exp{(2\pi\sqrt{-1}w)}\otimes \gamma ,
\end{equation*}
$q$ can be written as $q=\mb{e}(z)\cdot x_{\sigma (q)}$.
Here $\ker{(\mb{e})}=\log{(\Gamma(\sigma)^{\text{gp}})}$ and $z$ is determined uniquely modulo $\log{(\gs)}+\sigma(q)_{\CC}$.

We define the analytic space $\check{E}_{\sigma}:=\toric_{\sigma}\times \check{D}$ and the subset
\begin{equation*}
E_{\sigma}:=\left\{\begin{array}{l|l}(q,F)\in\check{E}_{\sigma}&\begin{array}{r}\exp{(\sigma(q)_{\CC})}\exp{(z)}F\text{ is $\sigma (q)$-nilpotent orbit}\\ \text{ where $q=\mb{e}(z)\cdot x_{\sigma (q)}$.}\end{array}\end{array}\right\}.
\end{equation*}
Here we endow $\es$ with the {\it strong topology} (\cite[\S 3.1]{ku}) in $\check{E}_{\sigma}$. 
We then define the canonical map
\begin{align*}\label{es}
\pi: &E_{\sigma}\to \Gamma(\sigma)^{\mathrm{gp}}\backslash D_{\sigma},\\
&(q,F)\mapsto (\sigma(q),\exp{(\sigma(q)_{\CC})}\exp{(z)}F)\mod{\Gamma(\sigma)^{\mathrm{gp}}}.
\end{align*}
We endow $\Gamma(\sigma)^{\gp}\backslash D_{\sigma}$ with the strongest topology for which the maps $\pi$ are continuous.
\cite{ku} gives the geometric properties of $\es$, $\gsds$ and $\torsor$ by using the language {\it log manifolds} (\cite[\S 3.5]{ku}):  
\begin{thm}[{\cite[Theorem A]{ku}}]\label{KU_A}
\begin{enumerate}
\item $E_{\sigma}$ and $\GsDs$ are logarithmic manifolds.
\item We have the $\sigma_{\CC}$-action on $E_{\sigma}$ over $\gsds$ by
$$a\cdot(q,F):=(\mb{e}(a)q,\exp{(-a)}F)\quad(a\in\sigma_{\CC},\;(q,F)\in \es),$$
and $\kutorsor$ is a $\sigma_{\CC}$-torsor in the category of logarithmic manifold.
\end{enumerate}
\end{thm}
Log manifolds are roughly analytic spaces with {\it slits}.
A typical example of log manifold is 
$$\CC^2\supset\{(x,y)\; |\;x=0\Rightarrow y=0\}$$
which is defined by the log differential 1-form $y d\log{x}$ of the log analytic space $\CC^2$.

Moreover \cite{ku} defines {\it polarized log Hodge structures} (\cite[\S 2.4]{ku}), and they show $\Gamma(\sigma)^{\gp}\backslash D_{\sigma}$ is a fine moduli space of polarized log Hodge structures (\cite[Theorem B]{ku}).
  
In the classical situation, $\gsds$ is just a toroidal partial compactification and the boundary is of codimension $1$ (see \cite{n}).
However, the codimension may be greater than $1$ in the non-classical situation.
\begin{exm}[The (1,1,1,1)-case]\label{1111}
Nilpotent orbits in the case where the Hodge numbers are $h^{3,0}=h^{0,3}=1$ and $h^{1,2}=h^{2,1}=1$, $0$ otherwise (we call it the (1,1,1,1)-case) are classified by \cite[\S 12.3]{ku} or \cite{GGK}.
In this case $D\cong Sp(2,\RR)/(U(1)\times U(1))$ and $\dim{D}=4$.
Here $D$ is not Hermitian symmetric space.
All possible nilpotent cones are of rank $1$.
For a nilpotent orbit $(\RR_{\geq 0} N, \exp{(\CC N)} F)$, we have the limiting mixed Hodge structure $(W(N), F)$ by \cite{s} twisting $W(N)$.
Here $(W(N),F)$ is one of the following types:
\begin{center}
\begin{tabular}{cc}
Type-I: $N^2=0,\;\dim{(\Im{N})}=1$.
&Type-II: $N^2=0,\;\dim{(\Im{N})}=2$.\\

$
\xymatrix{
&\stackrel{(2,2)}{\bullet}\ar@{->}[dd]^N& \\ \stackrel{(3,0)}{\bullet}& & \stackrel{(0,3)}{\bullet} \\ &\stackrel{(1,1)}{\bullet}&
}$&
$\xymatrix{
\stackrel{(3,1)}{\bullet}\ar@{->}[dd]^N&& \stackrel{(1,3)}{\bullet}\ar@{->}[dd]^N\\ \\\stackrel{(2,0)}{\bullet}&& \stackrel{(0,2)}{\bullet} 
}$\\
Type-III: $N^3\neq 0,N^4=0$.&
Dimensions of boundaries\\
$\xymatrix{
\stackrel{(3,3)}{\bullet}\ar@{->}[d]^N \\ \stackrel{(2,2)}{\bullet}\ar@{->}[d]^N\\ \stackrel{(1,1)}{\bullet} \ar@{->}[d]^N\\ \stackrel{(0,0)}{\bullet}
}$&
\begin{tabular}[c]{|c|c|}
\hline
&$\dim{(D_{\sigma}- D)}$\\ \hline
Type-I&$2$\\ \hline
Type-II&$1$\\ \hline
Type-III&$1$\\ \hline
\end{tabular}
\end{tabular}
\end{center}



Geometrically type-I or type-III degeneration occurs in the quintic-mirror family, and type-II degeneration occurs in the Borcea-Voisin mirror family (see \cite[Part III. A]{GGK}, \cite{U}). 
\end{exm}
\subsection{Whether the torsors are trivial}\label{main-sec}
By Theorem \ref{KU_A}, we have the torsor $\torsor$ for a period domain $D$ and a nilpotent cone $\sigma$.
In \cite{H}, we showed the triviality of torsors in the classical situation.
We show a non-triviality of the torsors in the non-classical situation by using the fact that any holomorphic functions on $D$ is constant in the non-classical case (Theorem \ref{const}).
\begin{thm}\label{no_global_sec}
Let $D$ be a period domain (for pure Hodge structures) and let $(\sigma, Z)$ be nilpotent orbit.
Then $\torsor$ is trivial if and only if $D$ is Hermitian symmetric or $\sigma=\{0\}$.
\end{thm}
\begin{proof}
By \cite[Theorem 5.6]{H} the torsors are trivial if $D$ is a Hermitian symmetric space.
If $\sigma=\{ 0\}$, the torsor is just the identity map $D\to D$, therefore the torsor is trivial.
Thus it is suffice to show that the torsor is non-trivial if $D$ is not Hermitian symmetric.

We assume that $\pi: \torsor$ is trivial for a non-Hermitian symmetric space $D$ and for a nilpotent cone $\sigma\neq \{0\}$.
Now 
$$\pi^{-1}(\Gs\bs D)=\Es\cap (\torus_{\sigma}\times \check{D})$$
by the definition of $\Es$, and this is a complex analytic space since $\torus_{\sigma}\times \check{D}$ has trivial log structure.
Thus the restriction of the torsor to $\pi^{-1}(\Gs\bs D)$ is a torsor in the category of complex analytic spaces, and we have a section $\gsds\to \Es$ and a holomorphic map $\Phi:D\to (\CC^*)^l$ such that we have the following diagram
\begin{align}
\xymatrix{
&\GsDs\ar@{}[d]|{\rotatebox{90}{$\subset$}}\ar@{->}[r]&\Es\ar@{}[d]|{\rotatebox{90}{$\subset$}}&\\
\Phi: D\ar@{->}[r]^{\text{quot.}} &\Gs\bs D\ar@{->}[r]&\Es\cap (\torus_{\sigma}\times \check{D})\ar@{->}[r]^{\text{proj.}}&\torus_{\sigma}\cong (\CC ^*)^l
}
\end{align}
where $l=\rank{\Gs}$.

For a nilpotent $N$ in the relative interior of $\sigma$, we have
\begin{align}\label{conv_nilp}
\lim_{y\to\infty}\exp{(iyN)}F = (\sigma ,\exp{(\sigma_{\CC})}F)
\end{align}
through $D\to\Gs\bs D\hookrightarrow \GsDs$ by \cite[Proposition 3.4.4]{ku}.
Then $\Phi(\exp{(iyN)}F)$ has to converge to $0\in \toric_{\sigma}$ as $y\to \infty$.
This contradicts Theorem \ref{const}.
\end{proof}

\section{Remarks on \cite{H}}\label{remH}
We showed the non-triviality of the torsor in \cite[Proposition 5.8]{H} using a different method from Theorem \ref{no_global_sec}.
We formulate it by using SL(2)-orbit theorem and cycle spaces and give a second proof of the non-triviality result for a special case.
While this second proof requires some special conditions, the result is stronger than the first one since it says there exists no section over certain open sets around a boundary point.
A property of some cycle spaces induces this result.
We observe the property of cycle spaces in the (1,1,1,1)-case explicitly in \S \ref{1111sec}.
In this section we assume $D$ is not Hermitian symmetric.
\subsection{SL(2)-orbits and cycle spaces}\label{sl2-cs}
Let $(\RR_{\geq 0}N,\exp{(\CC N)}F)$ be a nilpotent orbit.
By \cite{s} there exists the monodromy weight filtration $W(N)$ and $(W(N),F)$ is a mixed Hodge structure.
By \cite[Proposition 2.20]{CKS} there exists the $\RR$-split mixed Hodge structure $(W(N),\hat{F})$ associated to it.
We then have the Deligne decomposition $H_{\CC}=\bigoplus_{p,q}I^{p,q}$ for $(W(N),\hat{F})$ where
\begin{align*}
\hat{F}^p=\bigoplus_{r\leq p}I^{r,s},\quad W(N)_k=\bigoplus_{r+s=k}I^{r,s},\quad \overline{I^{p,q}}=I^{q,p}.
\end{align*}
By the SL(2)-orbit theorem (\cite[Theorem 5.13]{s}, \cite[\S 3]{CKS}), there exists the Lie group homomorphism $\rho: SL(2,\CC)\to G_{\CC}$ defined over $\RR$ and the holomorphic map $\phi:\PP^1\to \check{D}$ satisfying the following conditions:
\begin{enumerate}\label{sl2}
\item[(S1)] $\rho(g)\phi(z)=\phi(gz)$;
\item[(S2)] $\phi(0)=\hat{F}$;
\item[(S3)] $\rho_*(\nn_-)=N$;
\item[(S4)] $Hv=(p+q-w)v$ for $v\in I^{p,q}$ where $\rho_*(\hh)=H$;
\item[(S5)] $\rho_*:\mathfrak{sl}(2,\CC)\to\frakg_{\CC}$ is a $(0,0)$-morphism of Hodge structure where $\frakg_{\RR}$ (resp. $\mathfrak{sl}(2,\RR)$) has a Hodge structure of weight $0$ relative to $\phi(i)$ (resp. $i$),
\end{enumerate}
where $\{\nn_-,\hh,\nn_+,\}$ is the $sl_2$-triple (Example \ref{SL(2)}).

Let $F_0=\phi(i)$ be a base point of $D$.
We write
\begin{align}\label{def_X}
\rho_*(\nn_+)=N^+,\quad X=\frac{1}{2}(iN-H+iN^+).
\end{align}
Then $X\in\frakg^{-1,1}_{0}$ by (S5) and (\ref{sl2-triple}) where $\frakg_0^{-1,1}$ is the $(-1,1)$-component of Hodge decomposition of $\frakg_{\CC}$ with respect to $F_0$.
By (S1) we have
\begin{align*}
\exp{(zX)} \phi (i)=\phi \l(\frac{1+z}{1-z}i\r),
\end{align*}
therefore
\begin{align}\label{X-N}
\exp{\l(\frac{y}{2+y}X\r)}\phi (i)=\phi ((1+y)i)=\exp{(iyN)}\phi (i).
\end{align}

\begin{lem}
Let $C_0$ be the base cycle of $F_0$.
If $\dim{(\Im{N})}=1$, then both (A) and (B) hold:
\begin{enumerate}\label{fix}
\item[(A)] There exists $F_{\fix}\in C_0$ such that $\exp{(X)}F_{\fix}=F_{\fix}$;
\item[(B)] $\exp{(zX)}C_0\subset D$ (i.e., $\exp{(zX)}C_0\in\calM_{D}$) for $|z|<1$.
\end{enumerate}
\end{lem}
\begin{proof}
At first we write $X$ explicitly.
Considering the type of the limiting Hodge structure $(W(N),F)$, the case where $\dim{(\Im{N})}=1$ is possible only if the weight is $2m-1$ and $\dim{(\Gr^{W}_{2m})}=\dim{(\Gr^{W}_{2m-2})}=1$.
We then have a $\RR$-element $e$ in the $(m,m)$-component $I^{m,m}$ of the Deligne decomposition of $(W(N),\hat{F})$.
Here $X$ is given by
$$e\mapsto \frac{1}{2}(-e+iNe),\quad Ne\mapsto \frac{1}{2}(ie+Ne)=-iXe,\quad I^{p,q}\to 0\quad\text{for }p+q=2m-1.$$
We write $u=\exp{(iN)}e$.
Since $e\in \hat{F}^m$, $u \in \exp{(iN)}\hat{F}^m=F^m_{0}$.
Moreover, since $Ne\in\hat{F}^{m-1}$,
$$Ne=\exp{(iN)}Ne\in \exp{(iN)}\hat{F}^{m-1}=F^{m-1}_{0}.$$
Then 
$$\bar{u}=e-iNe=u-2iNe\in F^{m-1}_0.$$
Hence $u$ is in the $(m,m-1)$-component $H^{m,m-1}_0$ of the Hodge decomposition for $F_0$.
Here
\begin{align}\label{op-X}
Xu=-e+iNe=-\bar{u},\quad Xw=0\;\text{ if $\langle w,\bar{u}\rangle =0$}.
\end{align}

We show (A).
We denote by $\|\bullet \|$ the norm induced by the positive definite Hermitian form $\langle C_{F_0} \bullet ,\bar{\bullet}\rangle$ where $C_{F_0}$ is Weil operator for $F_0$.
Scaling $u$, we may assume $\| u\|=1$.
We take $v\in H^{m-2,m+1}_0$ such that $\| v\|=1$.
We define $g\in \Aut(H_{\CC})$ by
\begin{align*}
gu=v,\quad gv=u,\quad g\bar{v}=\bar{u},\quad g\bar{u}=\bar{v}, 
\end{align*}
and $gw=w$ if $w$ is vertical to $u,v,\bar{u}$ and $\bar{v}$ for $\langle \; ,\;\rangle $.
Then
$$gu=v=\overline{g\bar{u}}=\bar{g}u,\quad gv=u=\overline{g\bar{v}}=\bar{g}v.$$
Therefore $g$ is defined over $\RR$ and preserves the polarization $\langle \; ,\; \rangle$ i.e.\ $g\in G_{\RR}$.
Moreover $g\in K$ since $g$ preserves $H^{\even}$.\\
{\bf Claim.} $gF_0 \in C_0$ is a fixed point for $\exp{(X)}$.
\begin{proof}
Now $u=gv\in gH^{m-2,m+1}_0$.
By (\ref{op-X}) it is suffice to show that $Xu\in gF^{m-2}_0$.
In fact
$$Xu=-\bar{u}=-g\bar{v}\in gH^{m+1,m-2}_0.$$
\end{proof}

Next, we show (B).
We take a unitary basis $\{u_1,\ldots ,u_l\}$ of $H^{m,m-1}_0$.
We may assume $u_1=u$.
Then $\exp{(X)}u_j=u_j$ if $j\neq 1$ and $\exp{(zX)}u_1=u_1-z\bar{u}_1$.
Here 
$$i\langle\exp{(zX)}u_1,\overline{\exp{(zX)}u_1}\rangle=\|u_1\|^2-|z|^2\|u_1\|^2=1-|z|^2.$$
By Proposition \ref{cs}, $\exp{(zX)}C_0\subset D$ if and only if $|z|<1$.
\end{proof}
\begin{rem}
In the (1,1,1,1)-case, a type-I nilpotent $N$ satisfies $\dim(\Im{N})=1$, however other types do not.  
Above (\ref{X-N}), (A) and (B) are corresponding to the conditions (5.4), (5.6) and (5.5) of \cite[\S 5]{H} respectively. 
\end{rem}
\subsection{Non-triviality on some open sets}
Let $(\RR_{\geq 0}N,\exp{(\CC N)}F)$ be a nilpotent orbit.
Let $(\rho ,\phi)$ be the SL(2)-orbit associated to $(N,F)$.
Taking $F_0=\phi(i)$ as a base point, we have the base cycle $C_0$ and $X\in \frakg^{-1,1}_0$ as in (\ref{def_X}).
We define the subset
\begin{align*}
\calM (\varepsilon )=\{ \exp{(\alpha X)}C_0\; |\; 1-\varepsilon < \alpha <1  \} \subset \calM_{\check{D}}.
\end{align*}
for $0<\varepsilon$.
If $\dim{(\Im{N})}=1$, by Lemma \ref{fix} (B)
\begin{align*}
\exp{(\alpha X)}C_0\in \calM_{D}\text{ for }-1<\alpha <1,\quad \exp{(X)}C_0\notin \calM_{D}.
\end{align*}
Then $\calM (\varepsilon)$ is a nearby set of the boundary point $\exp{(X)}C_0\in \overline{\calM}_{D}$.   
\begin{prop}\label{no_loc_sec}
Let $U$ be an open set including the boundary point $(\sigma,\exp{(\sigma_{\CC})}\hat{F})$ in $\GsDs$ where $\sigma=\RR_{\geq 0}N$ with $\dim{(\Im{N})}=1$.
If there exists $0<\varepsilon<1$ such that $q(C)\subset U$ for any $C\in \calM(\varepsilon)$ and the quotient map $q:D\to \Gs\bs D$, then no section over the open set $U$ exists.
\end{prop}
\begin{proof}
We assume there exists a local trivialization over $U$.
Similar to the proof of Theorem \ref{no_global_sec}, we have a section $U\to E_{\sigma}$ and the holomorphic map $\Phi :q^{-1}(U)\to \CC^*$ given by the following diagram
\begin{align*}
\xymatrix{
&U\ar@{}[d]|{\rotatebox{90}{$\subset$}}\ar@{->}[r]&\Es\ar@{}[d]|{\rotatebox{90}{$\subset$}}&\\
\Phi: q^{-1}(U)\ar@{->}[r] &U\cap (\Gs\bs D) \ar@{->}[r]&\Es\cap (\CC^*\times \check{D})\ar@{->}[r]&\CC ^*.
}
\end{align*}

By (\ref{X-N}) and the assumption, we have
$$q\l(\exp{\l(\frac{y}{2+y} X\r)}F_0\r)=q(\exp{(iyN)}F_0)\subset U$$
for 
$$1-\varepsilon <\frac{y}{2+y}<1,\quad \text{i.e. }\frac{2 (1-\varepsilon) }{\varepsilon}\leq y.$$
By (\ref{conv_nilp}), $\Phi (\exp{(iyN)}F_0)$ has to converge to $0$ as $y\to \infty$.

Now $\Phi$ is constant on the compact complex submanifold $C\in \calM (\varepsilon)$.
By Lemma \ref{fix}, we then have
\begin{align*}
\Phi (\exp{(iyN)}F_{0})&=\Phi \l(\exp{\l(\frac{y}{2+y}X\r)}F_0\r)=\Phi \l(\exp{\l(\frac{y}{2+y}X\r)}F_{\fix}\r)\\
&=\Phi \l(\exp{\l(\frac{y'}{2+y'}X\r)}F_{\fix}\r)=\Phi (\exp{(iy'N)}F_{0})
\end{align*}
for $y,y'>2(1-\varepsilon) /\varepsilon $.
This contradicts the convergence of $\Phi (\exp{(iyN)}F_0)$.
\end{proof}

\begin{rem}
Above $X$, $F_{\fix}$ and $F_0$ are corresponding to the notations $N'$, $F_{\infty}$ and $F_0$ in \cite[\S 5]{H} respectively.
\end{rem}

\subsection{The (1,1,1,1)-case}\label{1111sec}
The conditions (A) and (B) of Lemma \ref{fix} induce Proposition \ref{no_loc_sec}.
In our later work \cite{H2}, the condition (B) also plays important role to study the boundary structure.
We then expect that $\Gamma\bs D_{\Sigma}$ has a good properties if $\Sigma$ satisfies the conditions (A) or (B).
Therefore it is important to determine which cone satisfies (A) or (B).
As we saw in Example \ref{1111}, the types of nilpotent orbits in the (1,1,1,1)-case are well-known.
Type-I nilpotents satisfies (A) and (B).
We show that (A) or (B) does not hold in other types below.

Let $(\RR_{\geq 0}N ,\exp{(\CC N)}F)$ be a nilpotent orbit and let $(\rho,\phi)$ be the SL(2)-orbit associated $(N,F)$.
We can choose a unitary basis
$$u_3\in H^{3,0}_0,\quad u_2\in H^{2,1}_0,\quad \bar{u}_2\in H^{1,2}_0 ,\quad \bar{u}_3\in H^{0,3}_0$$
for the Hodge decomposition for $F_0=\phi(i)$.
Here the base cycle of $F_0$ is $C_0\cong U(2)/(U(1)\times U(1))\cong \PP^1$.
The isomorphism $\PP^1\stackrel{\sim}{\to}C_0\subset D$ is given by 
\begin{align}\label{basecycle}
&F_z^3=\text{span}_{\CC}\{ z\bar{u}_2+u_3\},\quad F_z^2=\text{span}_{\CC}\{z\bar{u}_2+u_3,u_2-z\bar{u}_3\},\\
&F_{\infty} ^3=\text{span}_{\CC}\{\bar{u}_2\},\quad F_{\infty}^2=\text{span}_{\CC}\{\bar{u}_2,\bar{u}_3\}.\nonumber
\end{align}

The properties (A) and (B) of Lemma \ref{fix} depend on $X\in\frakg^{-1,1}_0$.
For a type-I nilpotent, $X$ is given by  
$$\xymatrix{
\stackrel{(3,0)}{\bullet} &\stackrel{(2,1)}{\bullet}\ar@{->}[r]^X&\stackrel{(1,2)}{\bullet}&\stackrel{(0,3)}{\bullet}
}\quad (u_2\mapsto -\bar{u}_2\mapsto 0).$$
We determine the type of $X$ in the case for type-II and for type-III, and consider whether (A) or (B) holds or not.
\subsubsection{type-II}
\begin{prop}
If $N$ is of type-II, then (B) holds, however (A) does not hold.
\end{prop}
\begin{proof}
Let $v$ be a non-zero element in $I^{3,1}$ of the Deligne decomposition of $(W(N),\hat{F})$.
Then 
$$Nv\in I^{2,0},\quad \bar{v}\in I^{1,3},\quad N\bar{v}\in I^{0,2}.$$
We write $u_3=\exp{(iN)}v$.
Since $v\in\hat{F}^3$, $u_3\in F^3_0=H^{3,0}_0$.
Here the $sl_2$-triple is given by
\begin{align*}
&N^+Nv=v,\quad N^+N\bar{v}=\bar{v},\quad N^+v=N^+\bar{v}=0,\\
&Hv=v,\quad H\bar{v}=\bar{v},\quad HNv=-Nv,\quad HN\bar{v}=-N\bar{v}.
\end{align*}
Then we have
\begin{align*}
H^{2,1}_0\ni Xu_3=-v+iNv=-\exp{(-iN)}v.
\end{align*}
We write $u_2=Xu_3$.
Then $Xu_2=0$.
Moreover $\bar{u}_2\in H^{1,2}_0$, and
\begin{align*}
X\bar{u}_2=\bar{v}-iN\bar{v}=\bar{u}_3.
\end{align*}
Summarizing these, $X\in\frakg_0^{-1,1}$ is given by
\begin{align*}
\xymatrix{
\stackrel{(3,0)}{\bullet} \ar@{->}[r]^X&\stackrel{(2,1)}{\bullet}&\stackrel{(1,2)}{\bullet}\ar@{->}[r]^X&\stackrel{(0,3)}{\bullet}
}\quad (u_3\mapsto {u}_2\mapsto 0,\;\bar{u}_2\mapsto \bar{u}_3\mapsto 0).
\end{align*}
Since $X(z\bar{u}_2+u_3)=z\bar{u}_3+u_2$, $XF_z^3\not\subset F_z^3$ for $z\in \PP^1$ in (\ref{basecycle}).
Then there is no fixed point for $\exp{(X)}$ in $C_0$.

Next we show (B) holds.
Scaling $v$, we may assume $\|u_3\| =1$.\\
{\bf Claim.} $\| u_2\| =1$.
\begin{proof}
Let $a=\langle v,\bar{v}\rangle$, $b=\langle Nv,\bar{v}\rangle$, $c=\langle v,N\bar{v}\rangle$ and $d=\langle Nv,N\bar{v}\rangle$.  
Then by the orthogonality
\begin{align*}
&\langle u_3,\bar{u}_3\rangle =a+ib-ic+d=i,\quad \langle u_3,\bar{u}_2\rangle =-a-ib-ic+d=0,\\
&\langle u_2, \bar{u}_3\rangle =-a+ib+ic+d=0.
\end{align*}
Since $v\in \hat{F}^3$ and $\bar{v}\in \hat{F}^1$, $a=0$.
Therefore the simultaneous equation induces $d=0$, $b-c=1$ and
$\langle u_2,\bar{u}_2\rangle =a-ib+ic+d=-i$.
\end{proof}
Here $\{u_3,u_2,\bar{u}_3,\bar{u}_2\}$ is a unitary basis.
Since
\begin{align*}
&-i\langle\exp{(zX)}u_3,\overline{\exp{(zX)}u_3}\rangle=\|u_3\|^2-|z|^2\|u_2\|^2=1-|z|^2,\\
&-i\langle\exp{(zX)}\bar{u}_2,\overline{\exp{(zX)}\bar{u}_2}\rangle=\|u_2\|^2-|z|^2\|u_3\|^2=1-|z|^2,
\end{align*}
$\exp{(zX)}C_0\subset D$ if and only if $|z|<1$ by Proposition \ref{cs}.
\end{proof}
\subsubsection{type-III}
We give an example of type-III which satisfies neither (A) nor (B).
All nilpotent orbits of type-III are described in \cite{GGK} explicitly.
We consider the case where $a,b=1$ and $e,f,\pi=0$ in the notation of \cite[(I.C.2), (I.C.10)]{GGK}.
Let $H_{\ZZ}=\sum_{j=0}^3 \ZZ e_j$.
We write 
\begin{align*}
e_3=\begin{pmatrix}
1\\0\\0\\0
\end{pmatrix},\quad
e_2=\begin{pmatrix}
0\\1\\0\\0
\end{pmatrix},\quad
e_1=\begin{pmatrix}
0\\0\\1\\0
\end{pmatrix},\quad
e_0=\begin{pmatrix}
0\\0\\0\\1
\end{pmatrix},
\end{align*}
where the bilinear form is given by
\begin{align*}
\begin{pmatrix}
0& 0&0&-1\\
0&0&-1&0\\
0&1&0&0\\
1&0&0&0
\end{pmatrix}.
\end{align*}
Let 
\begin{align*}
N=\begin{pmatrix}
0&0&0&0\\
1&0&0&0\\
0&1&0&0\\
0&0&-1&0\\
\end{pmatrix},\quad \hat{F}^p=\span\{ e_3,\ldots ,e_p \} (3\geq p\geq 0).
\end{align*}
Then $N$ and $\hat{F}$ give a nilpotent orbit of type-III, where the limit mixed Hodge structure $(W(N),\hat{F})$ is $\RR$-split.

The $sl_2$-triple of the SL(2)-orbit associated to this nilpotent orbit is given by
\begin{align*}
H=\begin{pmatrix}
3&0&0&0\\
0&1&0&0\\
0&0&-1&0\\
0&0&0&-3\\
\end{pmatrix},\quad
N^+=\begin{pmatrix}
0&3&0&0\\
0&0&4&0\\
0&0&0&-3\\
0&0&0&0\\
\end{pmatrix}.
\end{align*}
Then 
\begin{align*}
X=\frac{1}{2}\begin{pmatrix}
-3&3i&0&0\\
i&-1&4i&0\\
0&i&1&-3i\\
0&0&-i&3\\
\end{pmatrix}.
\end{align*}

\begin{prop}
For the above example, both (A) and (B) do not hold.
\end{prop}
\begin{proof}
Let
\begin{align*}
u_3=\frac{\sqrt{3}}{2}\exp{(iN)}e_3=\frac{\sqrt{3}}{12}\begin{pmatrix}
6\\6i\\-3\\i
\end{pmatrix}.
\end{align*}
Then $\|u_3\|=1$.
Now
\begin{align*}
&Xu_3=\frac{\sqrt{3}}{4}\begin{pmatrix}
-6\\-2i\\-1\\i
\end{pmatrix},\quad
X^2u_3=\frac{\sqrt{3}}{2}\begin{pmatrix}
6\\-2i\\1\\i
\end{pmatrix}=-2\overline{Xu_3},\\
&X^3u_3=\frac{\sqrt{3}}{2}\begin{pmatrix}
-6\\6i\\3\\i
\end{pmatrix}=-6\bar{u}_3.
\end{align*}
Here $\| Xu_3\| =3$.
Letting $u_2=\frac{1}{\sqrt{3}}Xu_3$, we then have a unitary basis $\{u_3,u_2,\bar{u}_3,\bar{u}_2\}$.
$X$ gives the map
\begin{align*}
\xymatrix{
\stackrel{(3,0)}{\bullet} \ar@{->}[r]^X&\stackrel{(2,1)}{\bullet}\ar@{->}[r]^X&\stackrel{(1,2)}{\bullet}\ar@{->}[r]^X&\stackrel{(0,3)}{\bullet}
}\\
(u_3\mapsto \sqrt{3}u_2 \mapsto -2\sqrt{3}\,\overline{u_2}\mapsto -6\,\overline{u_3}\mapsto 0).
\end{align*}
Then $XF^3_z\not\subset F^3_z$ for $z\in\PP^1$ in (\ref{basecycle}), and so there is no fixed point in $C_0$.
Moreover, for $\bar{u}_2\in F^3_{\infty}$
\begin{align*}
-i\langle\exp{(zX)}\bar{u}_2,\overline{\exp{(zX)}\bar{u}_2}\rangle=\|u_2\|^2-3|z|^2\|u_3\|^2=1-3|z|^2.
\end{align*}
Then $\exp{(zX)}C_0\not\subset D$ for $|z|\geq 1/\sqrt{3}$.
\end{proof}
\section*{Acknowledgment}
This work was done during a visit of the author to the Universit\'{e} du Qu\'{e}bec \`{a} Montr\'{e}al in August 2011 and the Johns Hopkins University in the activity of JAMI in September 2011. The author is grateful for the hospitality and the support.
The author is thankful to professors Patrick Brosnan, Steven Lu, Gregory Pearlstein and Steven Zucker for their valuable advice and warm encouragement. 
This research is supported by National Science Council of Taiwan.


\begin{thebibliography}{ABCD}
\bibitem[AMRT]{amrt} A. Ash, D. Mumford, M. Rapoport and Y. S. Tai, \it{Smooth compactification of locally symmetric varieties}, \em{Math.\ Sci.\ Press, Brookline}, 1975. 
\bibitem[CMP]{CMP}J. Carlson, S. M\:{u}ller-Stach and C. Peters, {\it Period mappings and period domains.} {\rm Cambridge Studies in Advanced Mathematics, {\bf 85}, Cambridge University Press, Cambridge, 2003.}
\bibitem[CKS]{CKS}E. Cattani, A. Kaplan and W. Schmid, {\it Degeneration of Hodge structures}, Ann. of Math. {\bf 123} (1986), 457--535. 
\bibitem[F]{F}  W.Fulton, {\it Introduction to toric varieties}, {\rm Annals of Mathematics Studies, {\bf 131}. Princeton University Press, Princeton, NJ, 1993.}
\bibitem[FHW]{FHW} G. Fels, A. Huckleberry and J. A. Wolf {\it Cycle Spaces of Flag Domains: A Complex Geometric Viewpoint}, {\rm Progress in Mathematics, {\bf 245}. Birkhauser Boston, Inc., 2006.} 
\bibitem[G]{G} P. Griffiths, {\it Periods of integrals on algebraic manifolds. I. Construction and properties of the modular varieties}, {\rm Amer. J. Math. {\bf 90} (1968) 568--626.}
\bibitem[GGK1]{GGK} M. Green, P. Griffiths and M. Kerr, {\it N\'eron models and boundary components for degenerations of Hodge structures of mirror quintic type}, {\rm in {\it Curves and Abelian Varieties} (V. Alexeev, Ed.), Contemp.\ Math {\bf 465} (2007), AMS, 71--145.}
\bibitem[GGK2]{GGK2} M. Green, P. Griffiths and M. Kerr, {\it Mumford-Tate groups and domains. Their geometry and arithmetic.} {\rm Annals of Mathematics Studies, {\bf 183}. Princeton University Press, Princeton, NJ, 2012.}
\bibitem[H1]{H} T. Hayama, {\it On the boundary of the moduli spaces of log Hodge structures: triviality of the torsor}, {\rm Nagoya Math.\ J.\ {\bf 198} (2010), 173--190}
\bibitem[H2]{H2} T. Hayama, {\it Boundaries of cycle spaces and degenerating Hodge structures}, {\rm  arXiv:1203.6770}
\bibitem[KU]{ku} K. Kato and S. Usui, {\it Classifying space of degenerating polarized Hodge structures}, {\rm  Annals of Mathematics Studies, {\bf 169}. Princeton University Press, Princeton, NJ, 2009.}
\bibitem[KP]{KP} G. Pearlstein and M.Kerr, {\it Boundary components of Mumford-Tate domains}, preprint.
\bibitem[LS]{LS} Z. Lu and X. Sun, {\it Weil-Petersson geometry on moduli space of polarized Calabi-Yau manifolds.} {\rm J. Inst. Math. Jussieu {\bf 3} (2004), no. 2, 185--229.}
\bibitem[N]{n} Y. Namikawa, {\it Toroidal compactification of Siegel spaces}, {\rm Lecture Notes in Math} {\bf 812}, {\rm Springer-Verlag}, 1980.
\bibitem[S]{s} W. Schmid, {\it Variation of Hodge structure: the singularities of the period mapping}, {\rm Invent.\ Math.\ {\bf 22} (1973), 211--319.}
\bibitem[U1]{U} S. Usui, {\it Generic Torelli theorem for quintic-mirror family}, {\rm Proc.\ Japan Acad.\ Ser.\ A Math.\ Sci.\ {\bf 84} (2008), no.\ 8, 143--146.}
\bibitem[U2]{U2} S. Usui, {\it Complex structures on partial compactifications of arithmetic quotients of classifying spaces of Hodge structures}, {\rm Tohoku Math. J. (2) {\bf 47} (1995), no. 3, 405--429.}
\end{thebibliography}
\end{document}